\documentclass[a4paper,12pt]{amsart}
\title[Congruence on tropical rational function semifield]{Congruences on tropical rational function semifields and tropical curves}

\author{Song JuAe}
\address{Department of Mathematics, Graduate School of Science, Kyoto University, Kitashirakawa Oiwake-cho, Sakyo-ku, Kyoto 606-8502, Japan.}
\email{song.juae.8m@kyoto-u.ac.jp}


\usepackage{amsmath,amssymb,amscd}
\usepackage{amsthm}
\usepackage{mathrsfs}

\newtheorem{dfn}{Definition}[section]
\newtheorem{thm}[dfn]{Theorem}
\newtheorem{prop}[dfn]{Proposition}
\newtheorem{cor}[dfn]{Corollary}
\newtheorem{lemma}[dfn]{Lemma}
\newtheorem{rem}[dfn]{Remark}
\newtheorem{ex}[dfn]{Example}

\def\Gamma{\varGamma}

\begin{document}

\maketitle

\begin{abstract}
We define tropical rational function semifields $\overline{\boldsymbol{T}(X_1, \ldots, X_n)}$ and prove that a tropical curve $\Gamma$ is realized (except for points at infinity) as the congruence variety $V \subset \boldsymbol{R}^n$ associated with a congruence on $\overline{\boldsymbol{T}(X_1, \ldots, X_n)}$ by giving a specific map $\Gamma \to V$.
Also, we shed light on the relation between congruences $E$ on $\overline{\boldsymbol{T}(X_1, \ldots, X_n)}$ and congruence varieties associated with them and reveal the quotient semifield $\overline{\boldsymbol{T}(X_1, \ldots, X_n)} / E$ to play the role of coordinate rings that determine isomorphism classes of affine varieties in the classical algebraic geometry.
\end{abstract}

\section{Introduction}
	\label{section1}

Congruences on the tropical polynomial semiring in $n$-variables $\boldsymbol{T}[X_1, \ldots, X_n]$ are central objects of interest in \cite{Bertram=Easton} and \cite{Joo=Mincheva2}.
Bertram and Easton (\cite{Bertram=Easton}) aimed to construct a general framework for algebraic geometry in the tropical setting.
They first studied basic properties of semirings and congruences on them, and then a strong form of tropical Nullstellensatz for congruences on $\boldsymbol{T}[X_1, \ldots, X_n]$.
\cite{Joo=Mincheva2} was devoted to studying prime congruences on polynomial semirings on additively idempotent semifields (containing $\boldsymbol{T}[X_1, \ldots, X_n]$) and to proving an improvement of the result of Bertram and Easton for tropical Nullstellensatz (for tropical Nullstellensatz, see also \cite{Grigoriev=Podolskii}).

The current paper also aims to develop an algebraic foundation for tropical geometry, but, in addition, is interested in the relationship between algebra and geometry; congruences and congruence varieties associated with them.
In this paper, we focus on congruences on the tropical rational function semifield in $n$-variables $\overline{\boldsymbol{T}(X_1, \ldots, X_n)}$, which is the fraction semifield of the tropical polynomial function semiring in $n$-variables.

One motivation to study congruences on $\overline{\boldsymbol{T}(X_1, \ldots, X_n)}$ is the rational function semifield $\operatorname{Rat}(\Gamma)$ of a tropical curve $\Gamma$.
By \cite[Theorem 1.1]{JuAe3}, $\operatorname{Rat}(\Gamma)$ is finitely generated as a semifield over $\boldsymbol{T}$.
Using this fact, we can make a surjective $\boldsymbol{T}$-algebra homomorphism $\psi : \overline{\boldsymbol{T}(X_1, \ldots, X_n)} \twoheadrightarrow \operatorname{Rat}(\Gamma)$.
By \cite[Theorem 1.1]{JuAe4}, the geometric structure of $\Gamma$ as a tropical curve (i.e., the topological structure and the metric structure) is defined by the $\boldsymbol{T}$-algebra structure of $\operatorname{Rat}(\Gamma)$, and hence the quotient semifield $\overline{\boldsymbol{T}(X_1, \ldots, X_n)} / \operatorname{Ker}(\psi)$ of $\overline{\boldsymbol{T}(X_1, \ldots, X_n)}$ by the kernel congruence $\operatorname{Ker}(\psi)$ of $\psi$.
Since the $\boldsymbol{T}$-algebra structure of $\overline{\boldsymbol{T}(X_1, \ldots, X_n)} / \operatorname{Ker}(\psi)$ is given by $\operatorname{Ker}(\psi)$, we expect that we can extract the geometric structure of $\Gamma$ as a tropical curve from $\operatorname{Ker}(\psi)$.
In fact, we can do it: Proposition \ref{prop2}.
This proposition states that $\Gamma$ is realized except for points at infinity as the congruence variety $\boldsymbol{V}(\operatorname{Ker}(\psi))$ associated with $\operatorname{Ker}(\psi)$.
Note that we consider $\boldsymbol{V}(\operatorname{Ker}(\psi))$ not just as a topological space but as a metric space with lattice length. 

Also by this proposition, we know that $\overline{\boldsymbol{T}(X_1, \ldots, X_n)} / \operatorname{Ker}(\psi)$ plays the role of coordinate rings that determine isomorphism classes of affine varieties in the classical algebraic geometry.
From this fact, we expect that $\overline{\boldsymbol{T}(X_1, \ldots, X_n)} / E$ also plays the role of coordinate rings for any congruence $E$ on $\overline{\boldsymbol{T}(X_1, \ldots, X_n)}$.
In fact, it does: Theorem \ref{thm1} and Proposition \ref{prop3}.

The rest of this paper is organized as follows.
In Section \ref{section2}, we give the definitions of semirings, algebras and semifields, congruences, tropical curves, rational functions and chip firing moves, and morphisms between tropical curves.
Section \ref{section3} is our main section; contains the statements of Propositions \ref{prop2}, \ref{prop3} and Theorem \ref{thm1} and their proofs.
In that section, we first give the definition of $\overline{\boldsymbol{T}(X_1, \ldots, X_n)}$ and study congruences on it and congruence varieties associated with them.
Then we prove the above three assertions.
To verify whether the conditions in our assertions are sufficient or necessary in those assertions, we give three examples: Examples \ref{ex1}, \ref{ex2}, \ref{ex3}.

\section*{Acknowledgements}
The author thanks my supervisor Masanori Kobayashi and Yasuhito Nakajima for their helpful comments.

\section{Preliminaries}
	\label{section2}

In this section, we recall several definitions which we need later.
We refer to \cite{Golan} (resp. \cite{Maclagan=Sturmfels}) for an introduction to the theory of semirings (resp. tropical geometry) and employ definitions in \cite{Bertram=Easton} and \cite{Joo=Mincheva2} (resp. \cite{JuAe3}) related to semirings (resp. tropical curves).
Also we employ the Chan's definition of morphisms between tropical curves in \cite{Chan}.

\subsection{Semirings, algebras and semifields}
	\label{subsection2.1}

In this paper, a \textit{semiring} $S$ is a commutative semiring with the absorbing neutral element $0_S$ for addition $+$ and the identity $1_S$ for multiplication $\cdot$.
If every nonzero element of a semiring $S$ is multiplicatively invertible and $0_S \not= 1_S$, then $S$ is called a \textit{semifield}.

A map $\varphi : S_1 \to S_2$ between semirings is a \textit{semiring homomorphism} if for any $x, y \in S_1$,
\begin{align*}
\varphi(x + y) = \varphi(x) + \varphi(y), \	\varphi(x \cdot y) = \varphi(x) \cdot \varphi(y), \	\varphi(0) = 0, \	\text{and}\	\varphi(1) = 1.
\end{align*}
Give a semiring homomorphism $\psi : S_1 \to S_2$, if both $S_1$ and $S_2$ are semifield and $\psi$ is injective, then $S_2$ is a \textit{semifield over $S_1$}.

Given a semiring homomorphism $\varphi : S_1 \to S_2$, we call the pair $(S_2, \varphi)$ (for short, $S_2$) a \textit{$S_1$-algebra}.
For a semiring $S_1$, a map $\psi : (S_2, \varphi) \to (S_2^{\prime}, \varphi^{\prime})$ between $S_1$-algebras is a \textit{$S_1$-algebra homomorphism} if $\psi$ is a semiring homomorphism and $\varphi^{\prime} = \psi \circ \varphi$.
When there is no confusion, we write $\psi : S_2 \to S_2^{\prime}$ simply.
A bijective $S_1$-algebra homomorphism $S_2 \to S_2^{\prime}$ is a \textit{$S_1$-algebra isomorphism}.
Then $S_2$ and $S_2^{\prime}$ are said to be \textit{isomorphic}.

The set $\boldsymbol{T} := \boldsymbol{R} \cup \{ -\infty \}$ with two tropical operations:
\begin{align*}
a \oplus b := \operatorname{max}\{ a, b \} \quad	\text{and} \quad a \odot b := a + b,
\end{align*}
where $a, b \in \boldsymbol{T}$ and $a + b$ stands for the usual sum of $a$ and $b$, becomes a semifield.
Here, for any $a \in \boldsymbol{T}$, we handle $-\infty$ as follows:
\begin{align*}
a \oplus (-\infty) = (-\infty) \oplus a = a \quad \text{and} \quad a \odot (-\infty) = (-\infty) \odot a = -\infty.
\end{align*}
This triple $(\boldsymbol{T}, \odot, \oplus)$ is called the \textit{tropical semifield}.
The subset $\boldsymbol{B} := \{ 0, -\infty \}$ of $\boldsymbol{T}$ becomes a semifield with tropical operations of $\boldsymbol{T}$ and is called the \textit{boolian semifield}.
A $\boldsymbol{B}$-algebra $S$ is said to be \textit{cancellative} if whenever $x \cdot y = x \cdot z$ for some $x, y, z \in S$, then either $x = 0_S$ or $y = z$.
If $S$ is cancellative, then we can define the semifield $Q(S)$ of fractions as with in the case of integral domains.
In this case, the map $S \to Q(S); x \mapsto x/1_S$ becomes an injective $\boldsymbol{B}$-algebra homomorphism.

The \textit{tropical polynomials} are defined as in the usual way and the set of all of tropical polynomials in $n$-variables is denoted by $\boldsymbol{T}[X_1, \ldots, X_n]$.
It becomes a semiring with two tropical operations and called the \textit{tropical polynomial semiring}.
By the follwing example, we know that tropical polynomial semirings are not cancellative:

\begin{ex}
\upshape{
Let $X := X_1$.
Then we have
\begin{align*}
&~(X \oplus 0) \odot (X^{\odot 2} \oplus (-2) \odot X \oplus 0)\\
=&~ X^{\odot 3} \oplus (-2) \odot X^{\odot 2} \oplus X \oplus X^{\odot 2} \oplus (-2) \odot X \oplus 0\\
=&~ X^{\odot 3} \oplus X^{\odot 2} \oplus X \oplus 0
\end{align*}
and
\begin{align*}
&~(X \oplus 0) \odot (X^{\odot 2} \oplus (-1) \odot X \oplus 0)\\
=&~ X^{\odot 3} \oplus (-1) \odot X^{\odot 2} \oplus X \oplus X^{\odot 2} \oplus (-1) \odot X \oplus 0\\
=&~ X^{\odot 3} \oplus X^{\odot 2} \oplus X \oplus 0.
\end{align*}
}
\end{ex}

A vector $\boldsymbol{v} \in \boldsymbol{R}^n$ is \textit{primitive} if all its components are integers and their greatest common divisor is one.
When $\boldsymbol{v}$ is primitive, for $\lambda \ge 0$, the \textit{lattice length} of $\lambda \boldsymbol{v}$ is defined as $\lambda$ (cf.~\cite[Subsection 2.1]{JuAe2}).

\subsection{Congruences}
	\label{subsection2.2}

A \textit{congruence} $E$ on a semiring $S$ is a subset of $S \times S$ satisfying

$(1)$ for any $x \in S$, $(x, x) \in E$,

$(2)$ if $(x, y) \in E$, then $(y, x) \in E$,

$(3)$ if $(x, y) \in E$ and $(y, z) \in E$, then $(x, z) \in E$,

$(4)$ if $(x, y) \in E$ and $(z, w) \in E$, then $(x + z, y + w) \in E$, and

$(5)$ if $(x, y) \in E$ and $(z, w) \in E$, then $(x \cdot z, y \cdot w) \in E$.

The diagonal set of $S \times S$ is denoted by $\Delta$ and called the \textit{trivial} congruence on $S$.
It is the unique smallest congruence on $S$.
The set $S \times S$ becomes a semiring with the operations of $S$ and is a congruence on $S$.
This is called the \textit{improper} congruence on $S$.
Congruences other than the improper congruence are said to be \textit{proper}.
Quotients by congruences can be considered in the usual sense and the quotient semiring of $S$ by the congruence $E$ is denoted by $S / E$.
Then the natural surjection $\pi_E : S \twoheadrightarrow S / E$ is a semiring homomorphism.

For a semiring homomorphism $\psi : S_1 \to S_2$, the \textit{kernel congruence} $\operatorname{Ker}(\psi)$ of $\psi$ is the congruence $\{ (x, y) \in S_1 \times S_1 \,|\, \psi(x) = \psi(y) \}$.
For semirings and congruences on them, the fundamental homomorphism theorem holds (\cite[Proposition 2.4.4]{Giansiracusa=Giansiracusa2}).
Then, for the above $\pi_E$, we have $\operatorname{Ker}(\pi_E) = E$.

\subsection{Tropical curves}
	\label{subsection2.3}

In this paper, a \textit{graph} is an unweighted, undirected, finite, connected nonempty multigraph that may have loops.
For a graph $G$, the set of vertices is denoted by $V(G)$ and the set of edges by $E(G)$.
A vertex $v$ of $G$ is a \textit{leaf end} if $v$ is incident to only one edge and this edge is not a loop.
A \textit{leaf edge} is an edge of $G$ incident to a leaf end.

A \textit{tropical curve} is the topological space associated with the pair $(G, l)$ of a graph $G$ and a function $l: E(G) \to {\boldsymbol{R}}_{>0} \cup \{\infty\}$, where $l$ can take the value $\infty$ only on leaf edges, by identifying each edge $e$ of $G$ with the closed interval $[0, l(e)]$.
The interval $[0, \infty]$ is the one-point compactification of the interval $[0, \infty)$.
We regard $[0, \infty]$ not just as a topological space but as an extended metric space.
The distance between $\infty$ and any other point is infinite.
When $l(e)=\infty$, the leaf end of $e$ must be identified with $\infty$.
If $E(G) = \{ e \}$ and $l(e)=\infty$, then we can identify either leaf ends of $e$ with $\infty$.
When a tropical curve $\Gamma$ is obtained from $(G, l)$, the pair $(G, l)$ is called a \textit{model} for $\Gamma$.
There are many possible models for $\Gamma$.
We frequently identify a vertex (resp. an edge) of $G$ with the corresponding point (resp. the corresponding closed subset) of $\Gamma$.
A model $(G, l)$ is \textit{loopless} if $G$ is loopless.
For a point $x$ of a tropical curve $\Gamma$, if $x$ is identified with $\infty$, then $x$ is called a \textit{point at infinity}, else, $x$ is called a \textit{finite point}.
Let $\Gamma_{\infty}$ denote the set of all points at infinity of $\Gamma$.
If $x$ is a finite point, then the \textit{valence} $\operatorname{val}(x)$ is the number of connected components of $U \setminus \{ x \}$ with any sufficiently small connected neighborhood $U$ of $x$; if $x$ is a point at infinity, then $\operatorname{val}(x) := 1$.
We construct a model $(G_{\circ}, l_{\circ})$ called the {\it canonical model} for $\Gamma$ as follows.
Generally, we define $V(G_{\circ}) := \{ x \in \Gamma \,|\, \operatorname{val}(x) \not= 2 \}$ except for the following two cases.
When $\Gamma$ is homeomorphic to a circle $S^1$, we define $V(G_{\circ})$ as the set consisting of one arbitrary point of $\Gamma$.
When $\Gamma$ has the pair $(T, l)$ as its model, where $T$ is a tree consisting of three vertices and two edges and $l(E(T)) = \{ \infty \}$, we define $V(G_{\circ})$ as the set of two points at infinity and any finite point of $\Gamma$.
The word ``an edge of $\Gamma$" means an edge of $G_{\circ}$.

\subsection{Rational functions and chip firing moves}
	\label{subsection2.4}

Let $\Gamma$ be a tropical curve.
A continuous map $f : \Gamma \to \boldsymbol{R} \cup \{ \pm \infty \}$ is a \textit{rational function} on $\Gamma$ if $f$ is a constant function of $-\infty$ or a piecewise affine function with integer slopes, with a finite number of pieces and that can take the values $\pm \infty$ only at points at infinity.
For a finite point $x$ of $\Gamma$ and a rational function $f \in \operatorname{Rat}(\Gamma) \setminus \{ -\infty \}$, for one outgoing direction at $x$, the \textit{outgoing slope} of $f$ at $x$ is $\frac{f(y) - f(x)}{\operatorname{dist}(x,y)}$, where $\operatorname{dist}(x, y)$ denotes the distance between $x$ and $y$, with a finite point $y$ in a sufficiently small neighborhood of $x$ in the direction.
If $x$ is a point at infinity, then we regard the outgoing slope of $f$ at $x$ as the slope of $f$ from $y$ to $x$ times minus one, where $y$ is a finite point on the leaf edge incident to $x$ such that $f$ has a constant slope on the interval $(y, x)$.
In both cases, the definition of outgoing slope of $f$ at $x$ in the direction is independent of the choice of $y$.
The point $x$ is a \textit{zero} (resp. \textit{pole}) of $f$ if the sign of the sum of outgoing slopes of $f$ at $x$ is positive (resp. negative).
If $x$ is a point at infinity, then we regard the outgoing slope of $f$ at $x$ as the slope of $f$ from $y$ to $x$ times minus one, where $y$ is a finite point on the leaf edge incident to $x$ such that $f$ has a constant slope on the interval $(y, x)$.
Let $\operatorname{Rat}(\Gamma)$ denote the set of all rational functions on $\Gamma$.
For rational functions $f, g \in \operatorname{Rat}(\Gamma)$ and a point $x \in \Gamma \setminus \Gamma_{\infty}$, we define
\begin{align*}
(f \oplus g) (x) := \operatorname{max}\{f(x), g(x)\} \quad \text{and} \quad (f \odot g) (x) := f(x) + g(x).
\end{align*}
We extend $f \oplus g$ and $f \odot g$ to points at infinity to be continuous on the whole of $\Gamma$.
Then both are rational functions on $\Gamma$.
Note that for any $f \in \operatorname{Rat}(\Gamma)$, we have
\begin{align*}
f \oplus (-\infty) = (-\infty) \oplus f = f
\end{align*}
and
\begin{align*}
f \odot (-\infty) = (-\infty) \odot f = -\infty.
\end{align*}
Then $\operatorname{Rat}(\Gamma)$ becomes a semifield with these two operations.
Also, $\operatorname{Rat}(\Gamma)$ becomes a $\boldsymbol{T}$-algebra and a semifield over $\boldsymbol{T}$ with the natural inclusion $\boldsymbol{T} \hookrightarrow \operatorname{Rat}(\Gamma)$.
Note that for $f, g \in \operatorname{Rat}(\Gamma)$, the equality $f = g$ means that $f(x) = g(x)$ for any $x \in \Gamma$.

A \textit{subgraph} of a tropical curve is a closed subset of the tropical curve with a finite number of connected components.
Let $\Gamma_1$ be a subgraph of a tropical curve $\Gamma$ which has no connected components consisting of only a point at infinity, and $l$ a positive number or infinity.
The \textit{chip firing move} by $\Gamma_1$ and $l$ is defined as the rational function $\operatorname{CF}(\Gamma_1, l)(x) := - \operatorname{min}\{ \operatorname{dist}(\Gamma_1, x), l \}$ with $x \in \Gamma$.

\subsection{Morphisms between tropical curves}
	\label{subsection2.5}

Let $\varphi : \Gamma \to \Gamma^{\prime}$ be a continuous map between tropical curves.
Then $\varphi$ is a \textit{morphism} if there exist loopless models $(G, l)$ and $(G^{\prime}, l^{\prime})$ for $\Gamma$ and $\Gamma^{\prime}$, respectively, such that $\varphi$ can be regarded as a map $V(G) \cup E(G) \to V(G^{\prime}) \cup E(G^{\prime})$ satisfying $\varphi(V(G)) \subset V(G^{\prime})$ and for $e \in \varphi(E(G))$, there exists a nonnegative integer $\operatorname{deg}_e(\varphi)$ such that for any points $x$ and $y$ of $e$, the equality $\operatorname{dist}_{\varphi(e)}(\varphi (x), \varphi (y)) = \operatorname{deg}_e(\varphi) \cdot \operatorname{dist}_e(x, y)$ holds, where $\operatorname{dist}_{\varphi(e)}(\varphi(x), \varphi(y))$ (resp. $\operatorname{dist}_e(x, y)$) denotes the distance between $\varphi(x)$ and $\varphi(y)$ in $\varphi(e)$ (resp. $x$ and $y$ in $e$).
A map between tropical curves $\varphi : \Gamma \to \Gamma^{\prime}$ is an \textit{isomorphism} if $\varphi$ is a bijective morphism and the inverse map $\varphi^{-1}$ of $\varphi$ is also  a morphism.
By this definition, $\varphi$ is an isomorphism if and only is $\varphi$ is continuous on the whole of $\Gamma$ and is a local isometry on $\Gamma \setminus \Gamma_{\infty}$.

\section{Main results}
	\label{section3}

In this section, tropical rational function semifields and congruences on them play a central role.
We start our consideration by defining tropical rational function semifields.

For $V \subset \boldsymbol{T}^n$, we define
\begin{align*}
\boldsymbol{E}(V)_0 := \{ (f, g) \in \boldsymbol{T}[X_1, \ldots, X_n] \times \boldsymbol{T}[X_1, \ldots, X_n] \,|\, \forall x \in V, f(x) = g(x) \}.
\end{align*}
Then $\boldsymbol{E}(V)_0$ is a congruence on $\boldsymbol{T}[X_1, \ldots, X_n]$.

\begin{lemma}
	\label{lemma0}
The quotient $\boldsymbol{T}[X_1, \ldots, X_n] / \boldsymbol{E}(\boldsymbol{T}^n)_0$ is cancellative.
\end{lemma}

\begin{proof}
It follows from \cite[Theorem 1]{Bertram=Easton} and \cite[Proposition 5.5 and Theorem 4.14(v)]{Joo=Mincheva2}.
\end{proof}

\begin{dfn}
	\label{dfn1}
\upshape{
Let $\overline{\boldsymbol{T}[X_1, \ldots, X_n]}$ denote $\boldsymbol{T}[X_1, \ldots, X_n] / \boldsymbol{E}(\boldsymbol{T}^n)_0$.
We call it the \textit{tropical polynomial function semiring}.
By Lemma \ref{lemma0}, the semifield of fractions of $\overline{\boldsymbol{T}[X_1, \ldots, X_n]}$ is given.
We call it the \textit{tropical rational function semifield} and write it $\overline{\boldsymbol{T}(X_1, \ldots, X_n)}$.
}
\end{dfn}

\begin{rem}
\upshape{
We employed the notation in \cite[3.4. Example]{Giansiracusa=Giansiracusa2} as our notation of the tropical polynomial function semiring $\overline{\boldsymbol{T}[X_1, \ldots, X_n]}$ in Definition \ref{dfn1}.
Other authors use other notations; e.g., $\operatorname{Poly}[\boldsymbol{T}^n]$ in \cite{Bertram=Easton} and $\boldsymbol{T}[\boldsymbol{X}^{\pm}]_{\operatorname{fcn}}$ (exactly, it is not the tropical polynomial function semiring and is the tropical Laurent polynomial semiring) in \cite{Ito}.
The notation of the tropical rational function semifield follows that of the tropical polynomial function semiring and is first defined in this paper.
}
\end{rem}

In what follows, by abuse of notation, we write the image of $X_i$ in $\overline{\boldsymbol{T}(X_1, \ldots, X_n)}$ as $X_i$.

\begin{lemma}
	\label{lemma2}
Let $E$ be a congruence on $\overline{\boldsymbol{T}(X_1, \ldots, X_n)}$.
Then the following are equivalent:

$(1)$ if $(f, -\infty) \in E$, then $f = -\infty$;

$(2)$ $\overline{\boldsymbol{T}(X_1, \ldots, X_n)} / E$ is a semifield; and

$(3)$ $E$ is a semifield.
\end{lemma}

\begin{proof}
The quotient $\overline{\boldsymbol{T}(X_1, \ldots, X_n)} / E$ is a semiring.
Let $\pi_E : \overline{\boldsymbol{T}(X_1, \ldots, X_n)} \twoheadrightarrow \overline{\boldsymbol{T}(X_1, \ldots, X_n)} / E$ be the natural surjective semiring homomorphism.
Remark that for $f, g \in \overline{\boldsymbol{T}(X_1, \ldots, X_n)}$, $(f, g) \in E$ if and only if $\pi_E(f) = \pi_E(g)$.

$(1) \Rightarrow (2)$ : by $(1)$, we have $\pi_E(0) \not= \pi_E(-\infty)$.
Let $f \in \overline{\boldsymbol{T}(X_1, \ldots, X_n)}$ be such that $\pi_E(f) \not= \pi_E(-\infty)$.
By $(1)$, $f \not= -\infty$, and thus $f^{\odot(-1)} \not= -\infty$.
Hence we have $\pi_E(f) \odot \pi_E(f^{\odot(-1)}) = \pi_E(f \odot f^{\odot(-1)}) = \pi_E(0) \not= \pi_E(-\infty)$.
This means that $\overline{\boldsymbol{T}(X_1, \ldots, X_n)} / E$ is a semifield.

$(2) \Rightarrow (1)$ : we assume that there exists $f \in \overline{\boldsymbol{T}(X_1, \ldots, X_n)} \setminus \{ -\infty \}$ such that $(f, -\infty) \in E$.
Since $(0, -\infty) = (f^{\odot(-1)} \odot f, f^{\odot(-1)} \odot (-\infty)) \in E$, for any $g \in \overline{\boldsymbol{T}(X_1, \ldots, X_n)}$, we have $(g, -\infty) = (g \odot 0, g \odot (-\infty)) \in E$.
This means that $\overline{\boldsymbol{T}(X_1, \ldots, X_n)} / E = \{ \pi_E(-\infty) \}$ and hence it is not a semifield.

$(1) \Rightarrow (3)$ : clearly $(0, 0) \not= (-\infty, -\infty)$.
For any $(f, g) \in E \setminus \{ (-\infty, -\infty) \}$, since $f \not= -\infty$ and $g \not= -\infty$ by $(1)$, we have $f^{\odot(-1)} \not= -\infty$ and $g^{\odot(-1)} \not= -\infty$.
Thus we have $(g^{\odot(-1)}, f^{\odot(-1)}) = (f \odot f^{\odot(-1)} \odot g^{\odot(-1)}, g \odot f^{\odot(-1)} \odot g^{\odot(-1)}) \in E$, and hence $(f^{\odot(-1)}, g^{\odot(-1)}) \in E$.
Since $(0, 0) = (f \odot f^{\odot(-1)}, g \odot g^{\odot(-1)})$, we have $(f, g)^{\odot(-1)} = (f^{\odot(-1)}, g^{\odot(-1)}) \in E$.
This means that $E$ is a semifield.

$(3) \Rightarrow (1)$ : we assume that there exists $f \in \overline{\boldsymbol{T}(X_1, \ldots, X_n)} \setminus \{ -\infty \}$ such that $(f, -\infty) \in E$.
Then, for any $g, h \in \overline{\boldsymbol{T}(X_1, \ldots, X_n)}$, since $(f \odot g, -\infty \odot h) = (f \odot g, -\infty) \not= (0, 0)$, this $(f, -\infty)$ has no inverse element for $\odot$ in $E$.
Hence $E$ is not a semifield.
\end{proof}

\begin{dfn}$($\upshape{cf.~}\cite[Subsection 3.1]{Bertram=Easton}$)$
\upshape{
For a congruence $E$ on $\overline{\boldsymbol{T}(X_1, \ldots, X_n)}$, we define
\begin{align*}
\boldsymbol{V}(E) := \{ x \in \boldsymbol{R}^n \,|\, \forall (f, g) \in E, f(x) = g(x)\}.
\end{align*}
We call $\boldsymbol{V}(E)$ the \textit{congruence variety} associated with $E$.
}
\end{dfn}

\begin{lemma}
	\label{lemma1}
Let $E$ be a congruence on $\overline{\boldsymbol{T}(X_1, \ldots, X_n)}$.
If there exist $f \in \overline{\boldsymbol{T}(X_1, \ldots, X_n)} \setminus \{ -\infty \}$ and $t \in \boldsymbol{T} \setminus \{ 0 \}$ such that $(f, f \odot t) \in E$, then $\boldsymbol{V}(E) = \varnothing$.
\end{lemma}

\begin{proof}
It is clear by the definition of congruence varieties.
\end{proof}

\begin{lemma}
	\label{lemma3}
Let $E$ be a congruence on $\overline{\boldsymbol{T}(X_1, \ldots, X_n)}$.
If $\overline{\boldsymbol{T}(X_1, \ldots, X_n)} / E$ is not a semifield over $\boldsymbol{T}$, then $\boldsymbol{V}(E) = \varnothing$.
\end{lemma}

\begin{proof}
If $\overline{\boldsymbol{T}(X_1, \ldots, X_n)} / E$ is not a semifield, then, by Lemma \ref{lemma2}, there exists $f \in \overline{\boldsymbol{T}(X_1, \ldots, X_n)} \setminus \{ -\infty \}$ such that $(f, -\infty) \in E$.
Since $(f, -\infty) = (f, f \odot (-\infty)) \in E$, by Lemma \ref{lemma1}, $\boldsymbol{V}(E) = \varnothing$.

If $\overline{\boldsymbol{T}(X_1, \ldots, X_n)} / E$ is not a semifield over $\boldsymbol{T}$ but a semifield, then there exist $t_1 \not= t_2 \in \boldsymbol{T}$ such that $(t_1, t_2) \in E$.
Without loss of generality, we can assume that $t_1 \not= -\infty$.
Then $(0, 0 \odot t_2 \odot t_1^{\odot (-1)}) = (0 \odot t_1 \odot t_1^{\odot (-1)}, 0 \odot t_2 \odot t_1^{\odot (-1)}) \in E$.
Since $t_2 \odot t_1^{\odot (-1)} \not= 0$, we have $\boldsymbol{V}(E) = \varnothing$ by Lemma \ref{lemma1}.
\end{proof}

For $(f_1, f_2), (g_1, g_2) \in \overline{\boldsymbol{T}(X_1, \ldots, X_n)} \times \overline{\boldsymbol{T}(X_1, \ldots, X_n)}$, we define the \textit{twisted product} $(f_1, f_2) \rtimes (g_1, g_2) := (f_1 \odot g_1 \oplus f_2 \odot g_2, f_1 \odot g_2 \oplus f_2 \odot g_1)$ (cf. \cite{Bertram=Easton}, \cite{Joo=Mincheva2}).
For congruences $E$ and $F$ on $\overline{\boldsymbol{T}(X_1, \ldots, X_n)}$, we define $E \rtimes F$ as the congruence on $\overline{\boldsymbol{T}(X_1, \ldots, X_n)}$ generated by $\{ (f_1, f_2) \rtimes (g_1, g_2) \,|\, (f_1, f_2) \in E, (g_1, g_2) \in F \}$, i.e., the smallest congruence containing $\{ (f_1, f_2) \rtimes (g_1, g_2) \,|\, (f_1, f_2) \in E, (g_1, g_2) \in F \}$.

\begin{prop}$($\upshape{cf.~}\cite[Lemma 3.1]{Bertram=Easton}$)$
	\label{prop0}
For $\lambda \in \Lambda$, let $E_{\lambda}$ be a semiring congruence on $\overline{\boldsymbol{T}(X_1, \ldots, X_n)}$.
Let $E$ be the congruence on $\overline{\boldsymbol{T}(X_1, \ldots, X_n)}$ generated by $\bigcup_{\lambda \in \Lambda} E_{\lambda}$.
Let $F_1$ and $F_2$ be congruences on $\overline{\boldsymbol{T}(X_1, \ldots, X_n)}$.
Then the following hold:

$(1)$ $\boldsymbol{V}(\Delta) = \boldsymbol{R}^n$ and $\boldsymbol{V}(\overline{\boldsymbol{T}(X_1, \ldots, X_n)} \times \overline{\boldsymbol{T}(X_1, \ldots, X_n)}) = \varnothing$, 

$(2)$ $\boldsymbol{V}(E) = \bigcap_{\lambda \in \Lambda} \boldsymbol{V}(E_{\lambda})$, and

$(3)$ $\boldsymbol{V}(F_1 \rtimes F_2) = \boldsymbol{V}(F_1) \cup \boldsymbol{V}(F_2)$.
\end{prop}

\begin{proof}
$(1)$ It is clear.

$(2)$ Since $E \supset E_{\lambda}$ for any $\lambda \in \Lambda$, we have $\boldsymbol{V}(E) \subset \bigcap_{\lambda \in \Lambda} \boldsymbol{V}(E_{\lambda})$.

Conversely, for any $x \in \bigcap_{\lambda \in \Lambda} \boldsymbol{V}(E_{\lambda})$, $\lambda \in \Lambda$ and $(f, g) \in E_{\lambda}$, we have $f(x) = g(x)$.
Hence, for any $(h_1, h_2) \in E$, we have $h_1(x) = h_2(x)$, and then $x \in \boldsymbol{V}(E)$.

$(3)$ We shall show that $\boldsymbol{V}(F_1 \rtimes F_2) \supset \boldsymbol{V}(F_1) \cup \boldsymbol{V}(F_2)$.
Let $x \in \boldsymbol{V}(F_1)$.
For any $(f_1, f_2) \in F_1$ and $(g_1, g_2) \in F_2$, since $f_1(x) = f_2(x)$, we have $f_1(x) \odot g_1(x) \oplus f_2(x) \odot g_2(x) = f_1(x) \odot g_2(x) \oplus f_2(x) \odot g_1(x)$.
Thus $x \in \boldsymbol{V}(F_1 \rtimes F_2)$.
Similarly, if $x \in \boldsymbol{V}(F_2)$, then $x \in \boldsymbol{V}(F_1 \rtimes F_2)$.

We shall show that $\boldsymbol{V}(F_1 \rtimes F_2) \subset \boldsymbol{V}(F_1) \cup \boldsymbol{V}(F_2)$.
Assume that there exists $f \in \overline{\boldsymbol{T}(X_1, \ldots, X_n)} \setminus \{ -\infty \}$ such that $(f, -\infty) \in F_1$.
By Lemmas \ref{lemma2}, \ref{lemma3}, we have $\boldsymbol{V}(F_1) = \varnothing$.
Let $x \in \boldsymbol{V}(F_1 \rtimes F_2)$.
For any $(g_1, g_2) \in F_2$, since $f(x) \odot g_1(x) \oplus (-\infty) \odot g_2(x) = f(x) \odot g_2(x) \oplus (-\infty) \odot g_1(x)$, $f(x) \in \boldsymbol{R}$ and  $g_1(x), g_2(x) \in \boldsymbol{T}$, we have $g_1(x) = g_2(x)$.
Therefore $x \in \boldsymbol{V}(F_2)$.
In the same way, we can show that if there exists $g \in \overline{\boldsymbol{T}(X_1, \ldots, X_n)} \setminus \{ -\infty \}$ such that $(g, -\infty) \in F_2$, then $\boldsymbol{V}(F_1 \rtimes F_2) \subset \boldsymbol{V}(F_1)$.
Assume that there are not $f, g \in \overline{\boldsymbol{T}(X_1, \ldots, X_n)} \setminus \{ -\infty \}$ such that $(f, -\infty) \in F_1$ and $(g, -\infty) \in F_2$.
Let $x \in \boldsymbol{V}(F_1 \rtimes F_2)$.
Assume that $x \not\in \boldsymbol{V}(F_1)$.
In this case, there exists $(f_1, f_2) \in F_1$ such that $f_1(x) \not= f_2(x)$.
By assumption, $f_1(x) \not= -\infty$ and $f_2(x) \not= -\infty$.
Without loss of generality, we can assume that $f_1(x) > f_2(x)$.
For any $(g_1, g_2) \in F_2$,
\begin{align*}
f_1(x) \odot g_1(x) \oplus f_2(x) \odot g_2(x) = f_1(x) \odot g_2(x) \oplus f_2(x) \odot g_1(x).
\end{align*}
If the left-hand side is $f_1(x) \odot g_1(x)$ and the right-hand side is $f_1(x) \odot g_2(x)$, then $g_1(x) = g_2(x)$ since $f_1(x) \not= -\infty$.
If the left-hand side is $f_1(x) \odot g_1(x)$ and the right-hand side is $f_2(x) \odot g_1(x)$, then $g_1(x) = -\infty$ since $f_1(x) > f_2(x) > -\infty$.
Thus $g_2(x) = -\infty$ in this case.
By the same argument, if the left-hand side is $f_2(x) \odot g_2(x)$, then $g_1(x) = g_2(x)$.
Hence $x \in \boldsymbol{V}(F_2)$.
\end{proof}

By Proposition \ref{prop0}, we can define the \textit{tropical Zariski topology} on $\boldsymbol{R}^n$ as in the classical way.

\begin{prop}
	\label{prop1}
The tropical Zariski topology coincides with the Euclidean topology on $\boldsymbol{R}^n$.
\end{prop}

\begin{proof}
By \cite[Lemma 3.7.4]{Giansiracusa=Giansiracusa}, the tropical Zariski topology defined by congruences on tropical polynomial semirings coincides with the Euclidean topology.
With this fact, the natural surjective $\boldsymbol{T}$-algebra homomorphism $\boldsymbol{T}[X_1, \ldots, X_n] \twoheadrightarrow \overline{\boldsymbol{T}[X_1, \ldots, X_n]}$ and the natural injective $\boldsymbol{T}$-algebra homomorphism $\overline{\boldsymbol{T}[X_1, \ldots, X_n]} \hookrightarrow \overline{\boldsymbol{T}(X_1, \ldots, X_n)}$, we can easily have the conclusion.
\end{proof}

Let $\Gamma$ be a tropical curve and let $f_1, \ldots, f_n \in \operatorname{Rat}(\Gamma) \setminus \{ -\infty \}$.
Note that, by \cite[Theorem 1.1]{JuAe3}, we can choose these $f_1, \ldots, f_n$ to generate $\operatorname{Rat}(\Gamma)$ as a semifield over $\boldsymbol{T}$, i.e., $\operatorname{Rat}(\Gamma) = \boldsymbol{T}(f_1, \ldots, f_n)$.

\begin{rem}
	\label{rem1}
\upshape{
When $\operatorname{Rat}(\Gamma) = \boldsymbol{T}(f_1, \ldots, f_n)$, for any $x \in \Gamma_{\infty}$, there exists a number $i$ such that $f_i(x) = \infty$ or $f_i(x) = -\infty$.
In fact, if it does not hold, then $f_j(x) \in \boldsymbol{R}$ for any $j$, and thus $f_1, \ldots, f_n$ cannot generate $\operatorname{Rat}(\Gamma)$.
}
\end{rem}

\begin{lemma}
	\label{lemma4}
The correspondence $X_i \mapsto f_i$ induces a $\boldsymbol{T}$-algeba homomorphism $\psi : \overline{\boldsymbol{T}(X_1, \ldots, X_n)} \to \operatorname{Rat}(\Gamma)$.
In particular, if $\operatorname{Rat}(\Gamma) = \boldsymbol{T}(f_1, \ldots, f_n)$, then $\psi$ is surjective.
\end{lemma}

\begin{proof}
Let $\psi_0 : \boldsymbol{T}[X_1, \ldots, X_n] \to \operatorname{Rat}(\Gamma)$ be the $\boldsymbol{T}$-algebra homomorphism induced by the correspondence $X_i \mapsto f_i$.
We shall show $\operatorname{Ker}(\psi_0) \supset \boldsymbol{E}(\boldsymbol{T}^n)_0$.
Let $(f, g) \in \boldsymbol{E}(\boldsymbol{T}^n)_0$.
For any $x \in \Gamma \setminus \Gamma_{\infty}$, since $f_i(x) \in \boldsymbol{R}$, we have
\begin{align*}
\psi_0(f)(x) =&~ f(f_1(x), \ldots, f_n(x))\\
=&~ g(f_1(x), \ldots, f_n(x))\\
=&~ \psi_0(g)(x).
\end{align*}
Hence for any $x \in \Gamma$, we have $\psi_0(f)(x) = \psi_0(g)(x)$.
Thus $\psi_0(f) = \psi_0(g)$, i.e., $(f, g) \in \operatorname{Ker}(\psi_0)$.
This $\psi_0$ induces the $\boldsymbol{T}$-algebra homomorphism $\psi_1 : \overline{\boldsymbol{T}[X_1, \ldots, X_n]} \to \operatorname{Rat}(\Gamma)$ satisfying $\pi \circ \psi_1 = \psi_0$, where $\pi$ stands for the natural surjective $\boldsymbol{T}$-algebra homomorphism $\boldsymbol{T}[X_1, \ldots, X_n] \twoheadrightarrow \overline{\boldsymbol{T}[X_1, \ldots, X_n]}$.
As $\overline{\boldsymbol{T}[X_1, \ldots, X_n]} \hookrightarrow \overline{\boldsymbol{T}(X_1, \ldots, X_n)}$ naturally, this $\psi_1$ is extended to a $\boldsymbol{T}$-algebra homomorphism $\overline{\boldsymbol{T}(X_1, \ldots, X_n)} \to \operatorname{Rat}(\Gamma)$.
It is $\psi$ we wanted.
When $\{ f_1, \ldots, f_n \}$ is a generating set of $\operatorname{Rat}(\Gamma)$, this $\psi$ is surjective.
\end{proof}

For $V \subset \boldsymbol{R}^n$, we define
\begin{align*}
\boldsymbol{E}(V) := \{ (f, g) \in \overline{\boldsymbol{T}(X_1, \ldots, X_n)} \times \overline{\boldsymbol{T}(X_1, \ldots, X_n)} \,|\, \forall x \in V, f(x) = g(x) \}.
\end{align*}
Then $\boldsymbol{E}(V)$ is a congruence on $\overline{\boldsymbol{T}(X_1, \ldots, X_n)}$.

Let $\theta : \Gamma \setminus \Gamma_{\infty} \to \boldsymbol{R}^n; x \mapsto (f_1(x), \ldots, f_n(x))$ and let $V := \boldsymbol{V}(\operatorname{Ker}(\psi))$.
Note that the image $\operatorname{Im}(\theta)$ of $\theta$ has the metric defined by the lattice length since rational functions on tropical curves have only integers as their slopes.

\begin{prop}
	\label{prop2}
In the above setting, the following hold:

$(1)$ for any $f \in \overline{\boldsymbol{T}(X_1, \ldots, X_n)}$ and $x \in \Gamma \setminus \Gamma_{\infty}$, $\psi(f)(x) = f (\theta(x))$ holds,

$(2)$ $\operatorname{Im}(\theta) \subset V$,

$(3)$ $\theta$ is continuous,

$(4)$ $\operatorname{Ker}(\psi) = \boldsymbol{E}(V)$, and

$(5)$ if $\psi$ is surjective, i.e., $\operatorname{Rat}(\Gamma) = \boldsymbol{T}(f_1, \ldots, f_n)$, then $\operatorname{Im}(\theta) \supset V$ and $\theta$ is an injective local isometry for the lattice length.
\end{prop}

\begin{proof}
$(1)$ Since $\psi$ is a $\boldsymbol{T}$-algebra homomorphism, by the definition of $\theta$, it is clear.

$(2)$ For any $(f, g) \in \operatorname{Ker}(\psi)$ and $x \in \Gamma \setminus \Gamma_{\infty}$, as $\psi(f) = \psi(g)$, by $(1)$, we have
\begin{align*}
f(\theta(x)) = \psi(f)(x) = \psi(g)(x) = g(\theta(x)).
\end{align*}
Thus $\theta(x) \in V$.

$(3)$ By Proposition \ref{prop1}, $f_i$ is continuous, and so is $\theta$.

$(4)$ Since $V = \boldsymbol{V}(\operatorname{Ker}(\psi))$, it is clear that $\operatorname{Ker}(\psi) \subset \boldsymbol{E}(V)$.
For any $(f, g) \in \boldsymbol{E}(V)$ and $x \in V$, the equality $f(x) = g(x)$ holds.
Hence, by $(2)$, for any $y \in \Gamma \setminus \Gamma_{\infty}$, we have $\psi(f)(y) = f(\theta(y)) = g(\theta(y)) = \psi(g)(y)$.
Thus $\psi(f)$ is eaqual to $\psi(g)$ as a rational function on $\Gamma$.
This means that $(f, g) \in \operatorname{Ker}(\psi)$.

$(5)$ We shall show that $\theta$ is injective.
For $x, y \in \Gamma \setminus \Gamma_{\infty}$, we assume $\theta(x) = \theta(y)$.
Since $\psi$ is surjective, there exists $f \in \overline{\boldsymbol{T}(X_1, \ldots, X_n)}$ such that $\psi(f) = \operatorname{CF}(\{ x \}, \infty)$.
Since
\begin{align*}
0 =&~ \operatorname{CF}(\{ x \}, \infty)(x) = \psi(f)(x) = f(\theta(x))\\
=&~ f(\theta(y)) = \psi(f)(y) = \operatorname{CF}(\{ x \}, \infty)(y)
\end{align*}
by $(1)$, $x$ must coincide with $y$.
Hence $\theta$ is injective.

Assume that the difference set $V \setminus \operatorname{Im}(\theta)$ is not empty.
For any $x \in V \setminus \operatorname{Im}(\theta)$, there exists $\varepsilon > 0$ such that the $(\varepsilon\sqrt{n})$-neighborhood of $x$ does not intersect $\operatorname{Im}(\theta)$.
In fact, if we cannot find such $\varepsilon$, this $x$ must be a boundary point of $\operatorname{Im}(\theta)$ in $\boldsymbol{R}^n$.
Since a connected subgraph $\Gamma_1$ of $\Gamma$ that contains no points at infinity is compact, so is $\theta(\Gamma_1)$.
Hence for $x$, there exists $y \in \Gamma_{\infty}$ such that any sequence $\{ y_i \}$ such that $y_i \to y$ and $\theta(y_i) \to x$ as $i \to \infty$.
By Remark \ref{rem1} and since $x \in \boldsymbol{R}^n$, it does not occur.
This means that such $x$ does not exist.
Let $x =: (a_1, \ldots, a_n)$.
For 
\begin{align*}
f := \left( \bigoplus_{i=1}^n \left( a_i^{\odot(-1)} \odot X_i \oplus \left( a_i^{\odot(-1)} \odot X_i \right)^{\odot(-1)} \right) \right)^{\odot(-1)} \oplus (-\varepsilon),
\end{align*}
$\psi(f)$ is the constant function of $-\varepsilon$ on $\Gamma \setminus \Gamma_{\infty}$.
Thus $\psi(f) = -\varepsilon \in \operatorname{Rat}(\Gamma)$.
Since $\psi(-\varepsilon) = -\varepsilon$, by $(4)$, we have $(f, -\varepsilon) \in \boldsymbol{E}(V)$.
Therefore, for any $y \in V$, $f(y) = -\varepsilon$ holds, but it is a contradiction as $f(x) = 0 \not= -\varepsilon$.
Hence we have $\operatorname{Im}(\theta) \supset V$.

We shall show that $\theta$ is a local isometry.
Let $(G, l)$ be the model for $\Gamma$ such that $V(G)$ consists of the vertices of the underlying graph of canonical model for $\Gamma$ and the zeros and poles of $f_1, \ldots, f_n$.
For any edge $e$ of $G$, the greatest common divisor of the slopes of $f_1, \ldots, f_n$ on $e$ must be one.
In face, if it is zero or at least two, then there exist no rational expressions of $f_1, \ldots, f_n$ which have one as their slopes on any segment in $e$.
On the other hand, for such a segment, the chip firing move by a finite point on the segment and a sufficiently small positive number has slope one on the segment.
This contradicts that $f_1, \ldots, f_n$ generate $\operatorname{Rat}(\Gamma)$ as a semifield over $\boldsymbol{T}$.
Therefore $\theta$ is a local isometry.
\end{proof}

\begin{cor}
For any $n \ge 2$ and any tropical curve $\Gamma$, the tropical rational function semifield $\overline{\boldsymbol{T}(X_1, \ldots, X_n)}$ is not isomorphic to $\operatorname{Rat}(\Gamma)$ as a $\boldsymbol{T}$-algebra.
\end{cor}

\begin{proof}
If there exists a tropical curve $\Gamma$ such that $\operatorname{Rat}(\Gamma)$ is isomorphic to $\overline{\boldsymbol{T}(X_1, \ldots, X_n)}$ via $\psi$ as a $\boldsymbol{T}$-algebra, then, since $\operatorname{Ker}(\psi) = \Delta$ and $\boldsymbol{V}(\Delta) = \boldsymbol{R}^n$, by Proposition \ref{prop2}$(5)$, $\Gamma \setminus \Gamma_{\infty}$ is homeomorphic to $\boldsymbol{R}^n$.
It is a contradiction.
\end{proof}

By the following example, we know that the converse of Proposition \ref{prop2}$(5)$ does not hold:

\begin{ex}
	\label{ex1}
\upshape{
Let $\Gamma := [-\infty, \infty]$.
We define $f_1, f_2 \in \operatorname{Rat}(\Gamma)$ as $f_1(t) := 0$ and $f_2(t) := t$ if $t \le 0$; $f_1(t) := t$ and $f_2(t) := 0$ if $0 < t \le 1$; $f_1(t) := 1$ and $f_2(t) := -t$ if $t \ge 1$.
Clearly $f_1$ and $f_2$ do not generate $\operatorname{Rat}(\Gamma)$ as a semifield over $\boldsymbol{T}$.
By Lemma \ref{lemma4}, the correspondence $X_i \mapsto f_i$ induces a $\boldsymbol{T}$-algebra homomorphism $\psi : \overline{\boldsymbol{T}(X_1, X_2)} \to \operatorname{Rat}(\Gamma)$.
Let $\theta : \Gamma \setminus \Gamma_{\infty} \to \boldsymbol{R}^2; x \mapsto (f_1(x), f_2(x))$.
Then, since $\operatorname{Im}(\theta)$ is closed, we can prove that $\operatorname{Im}(\theta) \supset \boldsymbol{V}(\operatorname{Ker}(\psi))$ (and hence $\operatorname{Im}(\theta) = \boldsymbol{V}(\operatorname{Ker}(\psi))$) as in the proof of Proposition \ref{prop2}$(5)$.
By definition, $\theta$ is an injective local isometry for the lattice length.
}
\end{ex}

We can consider any congruence instead of $\operatorname{Ker}(\psi)$ in Proposition \ref{prop2}.
The following theorem is our main theorem:

\begin{thm}
	\label{thm1}
Let $E$ be a congruence on $\overline{\boldsymbol{T}(X_1, \ldots, X_n)}$, $F$ a congruence on $\overline{\boldsymbol{T}(Y_1, \ldots, Y_m)}$, $V := \boldsymbol{V}(E)$ and $W := \boldsymbol{V}(F)$.
Let $\pi_E : \overline{\boldsymbol{T}(X_1, \ldots, X_n)} \twoheadrightarrow \overline{\boldsymbol{T}(X_1, \ldots, X_n)}  / E$ be the natural surjective $\boldsymbol{T}$-algebra homomorphism.
Let $\psi : \overline{\boldsymbol{T}(X_1, \ldots, X_n)} / E \to \overline{\boldsymbol{T}(Y_1, \ldots, Y_m)} / F$ be a $\boldsymbol{T}$-algebra homomorphism and $\theta : W \to \boldsymbol{T}^n; y \mapsto (\psi(\pi_E(X_1))(y), \ldots, \psi(\pi_E(X_n))(y))$.
Then the following hold:

$(1)$ for any $f \in \overline{\boldsymbol{T}(X_1, \ldots, X_n)}$ and $y \in W$, $\psi(\pi_E(f))(y) = \pi_E(f)(\theta(y)) = f(\theta(y))$ holds,

$(2)$ $\operatorname{Im}(\theta) \subset V$,

$(3)$ $\theta$ is continuous,

$(4)$ if $\psi$ is surjective, then $\theta$ is a closed embedding,

$(5)$ if $\psi$ is injective, $F = \boldsymbol{E}(W)$ and $\operatorname{Im}(\theta)$ is closed, then $\operatorname{Im}(\theta) \supset V$, and

$(6)$ if $\psi$ is an isomorphism and $F = \boldsymbol{E}(W)$, then $V$ and $W$ are homeomorphic via $\theta$.
\end{thm}

\begin{proof}
Let $\pi_F : \overline{\boldsymbol{T}(Y_1, \ldots, Y_m)} \twoheadrightarrow \overline{\boldsymbol{T}(Y_1, \ldots, Y_m)} / F$ be the natural surjective $\boldsymbol{T}$-algebra homomorphism.

$(1)$ Since $\psi$ is a $\boldsymbol{T}$-algebra homomorphism, by the definition of $\theta$, it is clear.

$(2)$ Let $(f, g) \in E$.
By $(1)$, for any $y \in W$, we have
\begin{align*}
f(\theta(y)) = \psi(\pi_E(f))(y) = \psi(\pi_E(g))(y) = g(\theta(y)).
\end{align*}
Thus $\theta(y) \in V$.

$(3)$ By $(2)$, $\operatorname{Im}(\theta) \subset V \subset \boldsymbol{R}^n$.
Hence, by Proposition \ref{prop1} and the definition of $\theta$, it is clear.

$(4)$ If $W$ is empty, then so is $\operatorname{Im}(\theta)$, and hence $\theta$ is a closed embedding.

Assume that $W$ is nonempty and $\psi$ is surjective.

We shall show that $\theta$ is injective.
For $x, y \in W$, we assume $\theta(x) = \theta(y)$ and let $x =: (a_1, \ldots, a_n)$.
Note that $a_i = \pi_F(Y_i)(x)$.
Since $\pi_E$ and $\psi$ are surjective, there exists $f \in \overline{\boldsymbol{T}(X_1, \ldots, X_n)}$ such that 
\begin{align*}
\psi(\pi_E(f)) = \left( \bigoplus_{i=1}^m \left( a_i^{\odot (-1)} \odot \pi_F(Y_i) \oplus \left( a_i^{\odot (-1)} \odot \pi_F(Y_i) \right)^{\odot (-1)} \right) \right)^{\odot (-1)}.
\end{align*}
Since
\begin{align*}
0 = \psi(\pi_E(f))(x) = f(\theta(x)) = f(\theta(y)) = \psi(\pi_E(f))(y)
\end{align*}
by $(1)$, $x$ must coincide with $y$.
Hence $\theta$ is injective.

We shall show that $\theta$ is a closed map.
Since $\varnothing \not= \operatorname{Im}(\theta) \subset V$ by $(2)$, $V$ is not empty.
Let $W_1$ be a closed subset of $W$.
If $W_1$ is compact, then so is $\theta(W_1)$ in $V \subset \boldsymbol{R}^n$.
Hence $\theta(W_1)$ is a closed subset of $V$ in this case.
If $W_1$ is not compact, then $W_1$ is unbounded.
Thus there exists a sequence $\{ y_k \}$ in $W_1$ such that $\pi_F(Y_j)(y_k) \to \infty$ or $-\infty$ as $k \to \infty$ for some $j$.
Since $\pi_E$ and $\psi$ are surjective, there exists $f \in \overline{\boldsymbol{T}(X_1, \ldots, X_n)}$ such that $\psi(\pi_E(f)) = \pi_F(Y_j)$.
For any such sequence $\{ y_k \}$, since $\pi_F(Y_j)(y_k) = \psi(\pi_E(f))(y_k) \to \infty$ or $-\infty$ as $k \to \infty$, there exists a number $l$ such that $\psi(\pi_E(X_l))(y_k) \to \infty$ or $-\infty$ as $k \to \infty$.
In fact, if for any $l^{\prime}$, $\psi(\pi_E(X_{l^{\prime}}))(y_k) \to b_{l^{\prime}} \in \boldsymbol{R}$ as $k \to \infty$, then 
\begin{align*}
\psi(\pi_E(f))(y_k) =&~ f(\theta(y_k)) \\
=&~ f(\psi(\pi_E(X_1))(y_k), \ldots, \psi(\pi_E(X_n))(y_k))\\
\to&~ b \in \boldsymbol{R}
\end{align*}
as $k \to \infty$ by $(1)$.
This is a contradiction.
Hence $\theta(W_1)$ is a closed set.

$(5)$ If $W$ is empty, then $F = \boldsymbol{E}(W) = \overline{\boldsymbol{T}(Y_1, \ldots, Y_m)} \times \overline{\boldsymbol{T}(Y_1, \ldots, Y_m)}$.
Hence $\overline{\boldsymbol{T}(Y_1, \ldots, Y_m)} / F = \{ \pi_F(-\infty) \}$.
Since $\psi$ is injective, $\overline{\boldsymbol{T}(X_1, \ldots, X_n)} / E = \{ \pi_E(-\infty) \}$.
By Lemma \ref{lemma3}, $V = \boldsymbol{V}(E)$ is also empty.
Hence we have the conclusion.

Assume that $W$ is not empty.
By $(2)$, we have $\varnothing \not= \operatorname{Im}(\theta) \subset V$.
By Lemma \ref{lemma3}, both $\overline{\boldsymbol{T}(X_1, \ldots, X_n)} / E$ and $\overline{\boldsymbol{T}(Y_1, \ldots, Y_m)} / F$ are semifields over $\boldsymbol{T}$.
Assume that the difference set $V \setminus \operatorname{Im}(\theta)$ is not empty.
Since $\operatorname{Im}(\theta)$ is closed, for any $x \in V \setminus \operatorname{Im}(\theta)$, there exists $\varepsilon > 0$ such that the $(\varepsilon \sqrt{n})$-neighborhood of $x$ does not intersect $\operatorname{Im}(\theta)$.
Let $x =: (a_1, \ldots, a_n)$.
For 
\begin{align*}
f := \left( \bigoplus_{i=1}^n \left( a_i^{\odot(-1)} \odot X_i \oplus \left( a_i^{\odot(-1)} \odot X_i \right)^{\odot(-1)} \right) \right)^{\odot(-1)} \oplus (-\varepsilon),
\end{align*}
$\psi(\pi_E(f))$ is the constant function of $-\varepsilon$ on $W$ by $(1)$.
On the other hand, since both $\overline{\boldsymbol{T}(X_1, \ldots, X_n)} / E$ and $\overline{\boldsymbol{T}(Y_1, \ldots, Y_m)} / F$ are semifields over $\boldsymbol{T}$, we have $\psi(\pi_E(-\varepsilon)) = -\varepsilon$, and thus it is also the constant function of $-\varepsilon$ on $W$.
As $F = \boldsymbol{E}(W)$, we have $\psi(\pi_E(f)) = \psi(\pi_E(-\varepsilon))$.
Since $\psi$ is injective, we have $\pi_E(f) = \pi_E(-\varepsilon)$.
However, $\pi_E(f)(x) = 0 \not= -\varepsilon$, and so $\pi_E(f) \not= \pi_E(-\varepsilon)$.
This is a contradiction.
Hence we have $V \setminus \operatorname{Im}(\theta) = \varnothing$, i.e., $\operatorname{Im}(\theta) \supset V$.

$(6)$ It follows from $(2), \ldots, (5)$.
\end{proof}

Theorem \ref{thm1} have the following corollaries:

\begin{cor}
In the same setting in Theorem \ref{thm1}, if $\psi$ is injective and $F = \boldsymbol{E}(W)$, then $E = \boldsymbol{E}(V)$.
In particular, when $\overline{\boldsymbol{T}(X_1, \ldots, X_n)} / E$ is isomorphic to $\overline{\boldsymbol{T}(Y_1, \ldots, Y_m)} / F$ as a $\boldsymbol{T}$-algebra, $E = \boldsymbol{E}(\boldsymbol{V}(E))$ if and only if $F = \boldsymbol{E}(\boldsymbol{V}(F))$.
\end{cor}

\begin{proof}
Assume that there exists $(f, g) \in \boldsymbol{E}(V) \setminus E$.
Since $\pi_E(f) \not= \pi_E(g)$ and $\psi$ is injective, we have $\psi(\pi_E(f)) \not= \psi(\pi_E(g))$.
Thus there exists $y \in W$ such that $\psi(\pi_E(f))(y) \not= \psi(\pi_E(g))(y)$.
By Theorem \ref{thm1}(1), we have $f(\theta(y)) = \psi(\pi_E(f))(y) \not= \psi(\pi_E(g))(y) = g(\theta(y))$, and this means $(f, g) \not\in \boldsymbol{E}(V)$ by Theorem \ref{thm1}(2), which is a contradiction.
\end{proof}

\begin{cor}[cf.~{\cite[Lemma 3.4]{JuAe5}}, {\cite[Lemma 3.7]{JuAe4}}]
	\label{cor1}
In the same setting in Theorem \ref{thm1}, for any $f \in \overline{\boldsymbol{T}(X_1, \ldots, X_n)} / E$, the following hold:

$(1)$ $\operatorname{sup}\{ f(x) \,|\, x \in V \} \ge \operatorname{sup}\{ \psi(f)(y) \,|\, y \in W \}$, and

$(2)$ $\operatorname{inf}\{ f(x) \,|\, x \in V \} \le \operatorname{inf}\{ \psi(f)(y) \,|\, y \in W \}$.

Moreover, if $W$ is nonempty, $\psi$ is injective and $F = \boldsymbol{E}(W)$, then the equalities hold in both $(1)$ and $(2)$.
\end{cor}

\begin{proof}
By Theorem \ref{thm1}(1), (2), we have
\begin{align*}
\{ f(x) \,|\, x \in V \} \supset \{ f(\theta(y)) \,|\, y \in W \} = \{ \psi(f)(y) \,|\, y \in W \}.
\end{align*}
Hence we have the conclusions.

Assume that $W$ is nonempty, $\psi$ is injective and $F = \boldsymbol{E}(W)$.
By Theorem \ref{thm1}(2), $\varnothing \not= \operatorname{Im}(\theta) \subset V$.
Since $V \not= \varnothing$ and $W \not= \varnothing$, by Lemma \ref{lemma3}, both $\overline{\boldsymbol{T}(X_1, \ldots, X_n)} / E$ and $\overline{\boldsymbol{T}(Y_1, \ldots, Y_m)} / F$ are semifields over $\boldsymbol{T}$.
Hence, for any $t \in \boldsymbol{T}$, we have $\psi(t) = t$.
Thus, if $f \in \boldsymbol{T}$, the assertions are clear.

Assume that $f$ is not a constant.
Let $a := \operatorname{sup}\{ f(x) \,|\, x \in V \}$.
In this case, since $f \not= -\infty$, by Lemma \ref{lemma2}, $a \not= -\infty$.
As $V \not= \varnothing$, $a$ is in $\boldsymbol{R} \cup \{ \infty \}$.
For $b \in \boldsymbol{R}$, as $a$ is the supremum of $f$ on $V$, we have $f \oplus b \not= b$ when $b < a$.
Therefore, since $\psi$ is injective, if $b < a$, then we have
\begin{align*}
\psi(f) \oplus b = \psi(f) \oplus \psi(b) = \psi(f \oplus b) \not= \psi(b) = b
\end{align*}
Thus, for any such $b$, there exists $y \in W$ such that $\psi(f)(y) \oplus b \not= b$ as $F = \boldsymbol{E}(W)$.
This means that $\psi(f)(y) > b$, and hence $a \le \operatorname{sup}\{ \psi(f)(y) \,|\, y \in W \}$.

For the infimums of $f$ and $\psi(f)$, we can obtain the conclusion by applying the supremum case for $f^{\odot (-1)} = -f$ and $\psi(f^{\odot (-1)}) = \psi(f)^{\odot (-1)} = -\psi(f)$ since 
\begin{align*}
\operatorname{inf}\{ f(x) \,|\, x \in V \} = - \operatorname{sup}\{ -f(x) \,|\, x \in V \}
\end{align*}
and 
\[
\operatorname{inf}\{ \psi(f)(y) \,|\, y \in W \} = - \operatorname{sup}\{ -\psi(f)(y) \,|\, y \in W \}. \qedhere
\]
\end{proof}

Now we can give another proof of \cite[Theorem 1.1 and Corollary 3.22]{JuAe5}:

\begin{cor}
Let $\Gamma_1$ and $\Gamma_2$ be tropical curves and $\psi : \operatorname{Rat}(\Gamma_1) \to \operatorname{Rat}(\Gamma_2)$ a $\boldsymbol{T}$-algebra homomorphism.
If $\psi$ is injective, then there exists a unique surjective morphism $\varphi : \Gamma_2 \twoheadrightarrow \Gamma_1$ such that for any $f \in \operatorname{Rat}(\Gamma_1)$, $\psi(f) = f \circ \varphi$.
\end{cor}

\begin{proof}
By \cite[Theorem 1.1]{JuAe3}, we can choose $f_1, \ldots, f_n \in \operatorname{Rat}(\Gamma_1) \setminus \{ -\infty \}$ and $g_1, \ldots, g_m \in \operatorname{Rat}(\Gamma_2) \setminus \{ -\infty \}$ such that $\operatorname{Rat}(\Gamma_1) = \boldsymbol{T}(f_1, \ldots, f_n)$ and $\operatorname{Rat}(\Gamma_2) = \boldsymbol{T}(g_1, \ldots, g_m)$.
By Lemma \ref{lemma4}, the correspondences $X_i \mapsto f_i$ and $Y_i \mapsto g_i$ induce the surjective $\boldsymbol{T}$-algebra homomorphisms 
\begin{align*}
\psi_1 : \overline{\boldsymbol{T}(X_1, \ldots, X_n)} \twoheadrightarrow \operatorname{Rat}(\Gamma_1)
\end{align*}
and
\begin{align*}
\psi_2 : \overline{\boldsymbol{T}(Y_1, \ldots, Y_m)} \twoheadrightarrow \operatorname{Rat}(\Gamma_2)
\end{align*}
respectively.
Let 
\begin{align*}
\pi_{\operatorname{Ker}(\psi_1)} : \overline{\boldsymbol{T}(X_1, \ldots, X_n)} \twoheadrightarrow \overline{\boldsymbol{T}(X_1, \ldots, X_n)} / \operatorname{Ker}(\psi_1)
\end{align*}
and
\begin{align*}
\pi_{\operatorname{Ker}(\psi_2)} : \overline{\boldsymbol{T}(Y_1, \ldots, Y_m)} \twoheadrightarrow \overline{\boldsymbol{T}(Y_1, \ldots, Y_m)} / \operatorname{Ker}(\psi_2)
\end{align*}
be the natural surjective $\boldsymbol{T}$-algebra homomorphisms.
Let
\begin{align*}
\overline{\psi_1} : \overline{\boldsymbol{T}(X_1, \ldots, X_n)} / \operatorname{Ker}(\psi_1) \to \operatorname{Rat}(\Gamma_1)
\end{align*}
and
\begin{align*}
\overline{\psi_2} : \overline{\boldsymbol{T}(Y_1, \ldots, Y_m)} / \operatorname{Ker}(\psi_2) \to \operatorname{Rat}(\Gamma_2)
\end{align*}
be the $\boldsymbol{T}$-algebra isomomorphisms induced by $\psi_1$ and $\psi_2$ respectively, i.e., $\overline{\psi_i} \circ \pi_{\operatorname{Ker}(\psi_i)} = \psi_i$.
By Proposition \ref{prop2}, $\operatorname{Ker}(\psi_i) = \boldsymbol{E}(\boldsymbol{V}(\operatorname{Ker}(\psi_i)))$ and 
\begin{align*}
\theta_1 : \Gamma_1 \setminus \Gamma_{1 \infty} \to \boldsymbol{V}(\operatorname{Ker}(\psi_1)); x \mapsto (f_1(x), \ldots, f_n(x))
\end{align*}
and
\begin{align*}
\theta_2 : \Gamma_2 \setminus \Gamma_{2 \infty} \to \boldsymbol{V}(\operatorname{Ker}(\psi_2)); x \mapsto (g_1(x), \ldots, g_m(x))
\end{align*}
are isometries.
By Theorem \ref{thm1}$(2)$, we have
\begin{align*}
\theta : \boldsymbol{V}(&\operatorname{Ker}(\psi_2)) &&\to &&\boldsymbol{V}(\operatorname{Ker}(\psi_1));\\
&y &&\mapsto &(&(\overline{\psi_2}^{-1} \circ \psi \circ \overline{\psi_1})(\pi_{\operatorname{Ker}(\psi_1)}(X_1))(y), \ldots,\\
&&&&&(\overline{\psi_2}^{-1} \circ \psi \circ \overline{\psi_1})(\pi_{\operatorname{Ker}(\psi_1)}(X_n))(y)).
\end{align*}
Since $\Gamma_i \not= \varnothing$, we have $\boldsymbol{V}(\operatorname{Ker}(\psi_i)) \not= \varnothing$.
As $\overline{\psi_2}^{-1} \circ \psi \circ \overline{\psi_1}$ is injective, by Corollary \ref{cor1}, we have
\begin{align*}
&~ \operatorname{sup}\{ \pi_{\operatorname{Ker}(\psi_1)}(X_i)(x) \,|\, x \in \boldsymbol{V}(\operatorname{Ker}(\psi_1)) \}\\
=&~ \operatorname{sup}\{ (\overline{\psi_2}^{-1} \circ \psi \circ \overline{\psi_1})(\pi_{\operatorname{Ker}(\psi_1)}(X_i))(y) \,|\, y \in \boldsymbol{V}(\operatorname{Ker}(\psi_2)) \}\\
\end{align*}
and
\begin{align*}
&~ \operatorname{inf}\{ \pi_{\operatorname{Ker}(\psi_1)}(X_i)(x) \,|\, x \in \boldsymbol{V}(\operatorname{Ker}(\psi_1)) \}\\
=&~ \operatorname{inf}\{ (\overline{\psi_2}^{-1} \circ \psi \circ \overline{\psi_1})(\pi_{\operatorname{Ker}(\psi_1)}(X_i))(y) \,|\, y \in \boldsymbol{V}(\operatorname{Ker}(\psi_2)) \}.
\end{align*}
By Remark \ref{rem1}, for any $x \in \Gamma_{1\infty}$, there exists $i$ such that $(\overline{\psi_1} \circ \pi_{\operatorname{Ker}(\psi_1)})(X_i)(x) = \infty$ or $-\infty$.
Hence 
\begin{align*}
\operatorname{sup}\{ (\overline{\psi_2}^{-1} \circ \psi \circ \overline{\psi_1})(\pi_{\operatorname{Ker}(\psi_1)}(X_i))(y) \,|\, y \in \boldsymbol{V}(\operatorname{Ker}(\psi_2)) \} = \infty
\end{align*}
or
\begin{align*}
\operatorname{inf}\{ (\overline{\psi_2}^{-1} \circ \psi \circ \overline{\psi_1})(\pi_{\operatorname{Ker}(\psi_1)}(X_i))(y) \,|\, y \in \boldsymbol{V}(\operatorname{Ker}(\psi_2)) \} = -\infty.
\end{align*}
Since $(\psi \circ \overline{\psi_1})(\pi_{\operatorname{Ker}(\psi_1)}(X_i))$ and $(\overline{\psi_2}^{-1} \circ \psi \circ \overline{\psi_1})(\pi_{\operatorname{Ker}(\psi_1)}(X_i))$ take the same value at the points correponding by $\theta_2$ by Proposition \ref{prop2}$(1)$, this means that there exists $y \in \Gamma_{2\infty}$ such that $(\psi \circ \overline{\psi_1}) (\pi_{\operatorname{Ker}(\psi_1)}(X_i))(y) = \infty$ or $-\infty$.
As $\theta_1^{-1}, \theta_2$ and $\theta$ are continuous, for arbitrary sequence $\{ y_k \}$ conversing to $y$, the sequence $\{ (\theta_1^{-1} \circ \theta \circ \theta_2) (y_k) \}$ converges to $x$.
Hence $\operatorname{Im}(\theta)$ is closed.
By Theorem \ref{thm1}$(5)$, $\theta$ is surjective.
Since $\theta_1^{-1}$ and $\theta_2$ are isometries and $\theta$ is continuous, $\theta_1^{-1} \circ \theta \circ \theta_2$ are surjective and continuous.
This $\theta_1^{-1} \circ \theta \circ \theta_2$ naturally extends to a surjective continuous map from $\Gamma_2$ to $\Gamma_1$ and it is the morphism $\varphi$ we want.
In fact, by the definition of $\theta_1^{-1} \circ \theta \circ \theta_2$, for any $y \in \Gamma_2 \setminus \Gamma_{2\infty}$ and $f \in \operatorname{Rat}(\Gamma_1)$, we have $\psi(f)(y) = f((\theta_1^{-1} \circ \theta \circ \theta_2)(y)) = f(\varphi(y))$, and thus for any $y \in \Gamma_2$, we have $\varphi(f)(y) = f(\varphi(y))$.
This means that $\psi(f) = f \circ \varphi$ holds.
Since $(\overline{\psi_2}^{-1} \circ \psi \circ \overline{\psi_1}) (\pi_{\operatorname{Ker}(\psi_1)}(X_i))$ is peacewise affine with integer slopes and $\theta_1^{-1}$ and $\theta_2$ are isometries, $\varphi$ is a morphism.
If there exists a morphism $\varphi^{\prime}$ satisfying $\psi(f) = f \circ \varphi^{\prime}$ for any $f \in \operatorname{Rat}(\Gamma_1)$, then, for any $y \in \Gamma_2 \setminus \Gamma_{2\infty}$, we have
\begin{align*}
0 =&~ \operatorname{CF}( \{ \varphi(y) \}, \infty) (\varphi(y))\\
=&~ (\operatorname{CF}(\{ \varphi(y) \}, \infty) \circ \varphi) (y)\\
=&~ \psi(\operatorname{CF}(\{ \varphi(y) \}, \infty))(y)\\
=&~ (\operatorname{CF}(\{ \varphi(y) \}, \infty) \circ \varphi^{\prime})(y)\\
=&~ \operatorname{CF}(\{ \varphi(y) \}, \infty)(\varphi^{\prime}(y)).
\end{align*}
As $\operatorname{CF}(\{ \varphi(y) \}, \infty)$ takes zero at and only at $\varphi(y)$, we have $\varphi(y) = \varphi^{\prime}(y)$.
Since both $\varphi$ and $\varphi^{\prime}$ are continuous, we have $\varphi = \varphi^{\prime}$.
\end{proof}

For a nonempty set $A$, we write the identity map of $A$ as $\operatorname{id}_A$.

By the following lemma and corollary, we know that we can drop the condition ``$F = \boldsymbol{E}(W)$" in Theorem \ref{thm1}(6):

\begin{lemma}
	\label{lemma5}
Let $E, F$ and $G$ be congruences on $\overline{\boldsymbol{T}(X_1, \dots, X_n)}$, $\overline{\boldsymbol{T}(Y_1, \dots, Y_m)}$ and $\overline{\boldsymbol{T}(Z_1, \dots, Z_l)}$ respectively.
Let $\pi_E : \overline{\boldsymbol{T}(X_1, \ldots, X_n)} \twoheadrightarrow \overline{\boldsymbol{T}(X_1, \ldots, X_n)} / E$ and $\pi_F : \overline{\boldsymbol{T}(Y_1, \ldots, Y_m)} \twoheadrightarrow \overline{\boldsymbol{T}(Y_1, \ldots, Y_m)} / F$ be the natural surjective $\boldsymbol{T}$-algebra homomorphisms.
Let $\psi : \overline{\boldsymbol{T}(X_1, \dots, X_n)} / E \to \overline{\boldsymbol{T}(Y_1, \dots, Y_m)} / F$ and $\phi : \overline{\boldsymbol{T}(Y_1, \dots, Y_m)} / F \to \overline{\boldsymbol{T}(Z_1, \dots, Z_l)} / G$ be $\boldsymbol{T}$-algebra homomorphisms.
For
\begin{align*}
&\theta_{\psi} : \boldsymbol{V}(F) \to \boldsymbol{V}(E); y \mapsto (\psi(\pi_E(X_1))(y), \ldots, \psi(\pi_E(X_n))(y)),\\
&\theta_{\phi} : \boldsymbol{V}(G) \to \boldsymbol{V}(F); z \mapsto (\phi(\pi_F(Y_1))(z), \ldots, \phi(\pi_F(Y_m))(z))
\end{align*}
and
\begin{align*}
\theta_{\phi \circ \psi} : \boldsymbol{V}(G) \to \boldsymbol{V}(E); z \mapsto ((\phi \circ \psi)(\pi_E(X_1))(z), \ldots, (\phi \circ \psi)(\pi_E(X_n))(z)),
\end{align*}
the following hold:

$(1)$ $\theta_{\phi \circ \psi} = \theta_{\psi} \circ \theta_{\phi}$, and

$(2)$ if $\overline{\boldsymbol{T}(X_1, \ldots, X_n)} / E = \overline{\boldsymbol{T}(Z_1, \ldots, Z_l)} / G$ and $\phi \circ \psi = \operatorname{id}_{\overline{\boldsymbol{T}(X_1, \ldots, X_n)} / E}$, then $\theta_{\phi \circ \psi} = \operatorname{id}_{\boldsymbol{V}(E)}$.
\end{lemma}

\begin{proof}
$(1)$ For any $z \in \boldsymbol{V}(G)$, we have
\begin{align*}
(\theta_{\psi} \circ \theta_{\phi})(z) =&~ \theta_{\psi}(\theta_{\phi}(z))\\
=&~ \theta_{\psi}(\phi(\pi_F(Y_1))(z), \ldots, \phi(\pi_F(Y_m))(z))\\
=&~ (\psi(\pi_E(X_1))(\phi(\pi_F(Y_1))(z), \ldots, \phi(\pi_F(Y_m))(z)), \ldots,\\
&~ \psi(\pi_E(X_n))(\phi(\pi_F(Y_1))(z), \ldots, \phi(\pi_F(Y_m))(z)))\\
=&~ (\phi(\psi(\pi_E(X_1)))(z), \ldots, \phi(\psi(\pi_E(X_n)))(z))\\
=&~ ((\phi \circ \psi)(\pi_E(X_1))(z), \ldots, (\phi \circ \psi)(\pi_E(X_n))(z))\\
=&~ \theta_{\phi \circ \psi} (z).
\end{align*}

$(2)$ It is clear.
\end{proof}

\begin{cor}
	\label{cor2}
In the same setting in Theorem \ref{thm1}, if $\psi$ is an isomorphism, then $V$ and $W$ are homeomorphic via $\theta$.
\end{cor}

\begin{proof}
By Theorem \ref{thm1}(4) and Lemma \ref{lemma5}, we have the conclusion.
\end{proof}

In Theorem \ref{thm1}, we began our consideration from assuming congruences are given.
In the following proposition, we will assume subsets of Euclidean spaces are given.
For $V \subset \boldsymbol{R}^n$, let $F(V, \boldsymbol{T})$ denote the set of all functions $V \to \boldsymbol{T}$.
This $F(V, \boldsymbol{T})$ becomes a $\boldsymbol{T}$-algebra with natural operations.

\begin{prop}
	\label{prop3}
Let $V$ be a subset of $\boldsymbol{R}^n$ and $W$ a nonempty subset of $\boldsymbol{R}^m$.
For $f_1, \ldots, f_n \in \overline{\boldsymbol{T}(Y_1, \ldots, Y_m)} / \boldsymbol{E}(W)$, assume that the image of $\theta : W \to \boldsymbol{T}^n; y \mapsto (f_1(y), \ldots, f_n(y))$ is contained in $V$.
Let $\theta^{\ast} : \overline{\boldsymbol{T}(X_1, \ldots, X_n)} / \boldsymbol{E}(V) \to F(W, \boldsymbol{T}); f \mapsto f \circ \theta$.
Then the following hold:

$(1)$ $V \not= \varnothing$,

$(2)$ $\operatorname{Im}(\theta^{\ast}) \subset \overline{\boldsymbol{T}(Y_1, \ldots, Y_m)} / \boldsymbol{E}(W) \subset F(W, \boldsymbol{T})$,

$(3)$ $\theta^{\ast}$ is a $\boldsymbol{T}$-algebra homomorphism,

$(4)$ if $\operatorname{Im}(\theta) \supset V$, then $\theta^{\ast}$ is injective, and

$(5)$ if $V$ and $W$ are homeomorphic via $\theta$ and there exist $g_1, \ldots, g_m \in \overline{\boldsymbol{T}(X_1, \dots, X_n)} / \boldsymbol{E}(V)$ such that for any $x \in V$, $\theta^{-1} (x) = (g_1(x), \ldots, g_m(x))$, then $\overline{\boldsymbol{T}(X_1, \ldots, X_n)} / \boldsymbol{E}(V)$ and $\overline{\boldsymbol{T}(Y_1, \ldots, Y_m)} / \boldsymbol{E}(W)$ are isomorphic via $\theta^{\ast}$.
\end{prop}

\begin{proof}
$(1)$ Since $W \not= \varnothing$, we have $\operatorname{Im}(\theta) \not= \varnothing$.
By assumption, $\varnothing \not= \operatorname{Im}(\theta) \subset V$.

$(2)$ As any element $f$ of $\overline{\boldsymbol{T}(Y_1, \ldots, Y_m)} / \boldsymbol{E}(W)$ defines a function $W \to \boldsymbol{T}$ and for $f, g \in \overline{\boldsymbol{T}(Y_1, \ldots, Y_m)} / \boldsymbol{E}(W)$, if $f(y) = g(y)$ for any $y \in W$, then $f = g$ as elements of $\overline{\boldsymbol{T}(Y_1, \ldots, Y_m)} / \boldsymbol{E}(W)$, hence the natural map $\overline{\boldsymbol{T}(Y_1, \ldots, Y_m)} / \boldsymbol{E}(W) \to F(W, \boldsymbol{T})$ is injective.

For any $f \in \overline{\boldsymbol{T}(X_1, \ldots, X_n)} / \boldsymbol{E}(V)$ and $y \in W$, since
\begin{align*}
\theta^{\ast}(f)(y) =&~ (f \circ \theta)(y)\\
=&~ f(\theta(y))\\
=&~ f(f_1(y), \ldots, f_n(y))\\
=&~ f(f_1, \ldots, f_n)(y),
\end{align*}
we have $\theta^{\ast}(f) = f(f_1, \ldots, f_n) \in \overline{\boldsymbol{T}(Y_1, \ldots, Y_m)} / \boldsymbol{E}(W)$.

$(3)$ By $(1)$, $V$ is nonempty.
By assumption, $W$ is also nonempty.
Since $\varnothing \not= V \subset \boldsymbol{V}(\boldsymbol{E}(V))$ and $\varnothing \not= W \subset \boldsymbol{V}(\boldsymbol{E}(W))$, by Lemma \ref{lemma3}, both $\overline{\boldsymbol{T}(X_1, \ldots, X_n)} / \boldsymbol{E}(V)$ and $\overline{\boldsymbol{T}(Y_1, \ldots, Y_m)} / \boldsymbol{E}(W)$ are semifields over $\boldsymbol{T}$.
For any $t \in \boldsymbol{T}$, $\theta^{\ast}(t) = t \in \overline{\boldsymbol{T}(Y_1, \ldots,Y_m)} / \boldsymbol{E}(W)$ is clear.
For any $f, g \in \overline{\boldsymbol{T}(X_1, \ldots, X_n)} / \boldsymbol{E}(V)$, we have
\begin{align*}
\theta^{\ast}(f \oplus g) = (f \oplus g) \circ \theta = (f \circ \theta) \oplus (g \circ \theta) = \theta^{\ast}(f) \oplus \theta^{\ast}(g)
\end{align*}
and
\begin{align*}
\theta^{\ast}(f \odot g) = (f \odot g) \circ \theta = (f \circ \theta) \odot (g \circ \theta) = \theta^{\ast}(f) \odot \theta^{\ast}(g).
\end{align*}
Thus $\theta^{\ast}$ is a $\boldsymbol{T}$-algebra homomorphism.

$(4)$ For $f, g \in \overline{\boldsymbol{T}(X_1, \ldots, X_n)} / \boldsymbol{E}(V)$, we assume that $\theta^{\ast}(f) = \theta^{\ast}(g)$ holds.
Since $\operatorname{Im}(\theta) \supset V$, for any $x \in V$, there exists $y \in W$ such that $x = \theta(y)$.
Hence we have 
\begin{align*}
f(x) =&~ f(\theta(y)) = (f \circ \theta)(y) = \theta^{\ast}(f)(y)\\
=&~ \theta^{\ast}(g)(y) = (g \circ \theta)(y) = g(\theta(y)) = g(x).
\end{align*}
Therefore $f = g \in \overline{\boldsymbol{T}(X_1, \ldots, X_n)} / \boldsymbol{E}(V)$.

$(5)$ By assumption and $(2)$, the image of $(\theta^{-1})^{\ast} : \overline{\boldsymbol{T}(Y_1, \ldots, Y_m)} / \boldsymbol{E}(W) \to F(V, \boldsymbol{T}); g \mapsto g \circ \theta^{-1}$ is contained in $\overline{\boldsymbol{T}(X_1, \ldots, X_n)} / \boldsymbol{E}(V)$.
For any $f \in \overline{\boldsymbol{T}(X_1, \ldots, X_n)} / \boldsymbol{E}(V)$, we have
\begin{align*}
((\theta^{-1})^{\ast} \circ \theta^{\ast})(f) =&~ (\theta^{-1})^{\ast}(\theta^{\ast}(f)) = (\theta^{-1})^{\ast}(f \circ \theta)\\
=&~ (f \circ \theta) \circ \theta^{-1} = f \circ (\theta \circ \theta^{-1}) = f \circ \operatorname{id}_{V} = f.
\end{align*}
For any $g \in \overline{\boldsymbol{T}(Y_1, \ldots, Y_m)} / \boldsymbol{E}(W)$, we have
\begin{align*}
(\theta^{\ast} \circ (\theta^{-1})^{\ast})(g) =&~ \theta^{\ast}((\theta^{-1})^{\ast}(g)) = \theta^{\ast}(g \circ \theta^{-1})\\
=&~ (g \circ \theta^{-1}) \circ \theta = g \circ (\theta^{-1} \circ \theta) = g \circ \operatorname{id}_{W} = g.
\end{align*}
Thus we have the conclusion.
\end{proof}

By the following example, we know that we cannot drop the condition ``$\operatorname{Im}(\theta)$ is closed" in Theorem \ref{thm1}(5) and that the converse of Proposition \ref{prop3}(4) does not hold:

\begin{ex}
	\label{ex2}
\upshape{
Let $V := [0, 1] \subset \boldsymbol{R}$ and $W := \{ (t, 1 / t) \in \boldsymbol{R}^2 \,|\, 0 < t \le 1 \}$.
Since $V$ and $W$ are closed subsets of $\boldsymbol{R}$ and $\boldsymbol{R}^2$ respectively, by Proposition \ref{prop1}, these are congruence varieties.
Let $E := \boldsymbol{E}(V) = \{ (f, g) \in \overline{\boldsymbol{T}(X)} \times \overline{\boldsymbol{T}(X)} \,|\, \forall x \in V, f(x) = g(x) \}$ and $F := \boldsymbol{E}(W) = \{ (f, g) \in \overline{\boldsymbol{T}(Y_1, Y_2)} \times \overline{\boldsymbol{T}(Y_1, Y_2)} \,|\, \forall y \in W, f(y) = g(y) \}$.
Then $V = \boldsymbol{V}(E)$ and $W = \boldsymbol{V}(F)$.
Let $\pi_E : \overline{\boldsymbol{T}(X)} \twoheadrightarrow \overline{\boldsymbol{T}(X)} / E$ and $\pi_F : \overline{\boldsymbol{T}(Y_1, Y_2)} \twoheadrightarrow \overline{\boldsymbol{T}(Y_1, Y_2)} / F$ be the natural surjective $\boldsymbol{T}$-algebra homomorphisms and $\theta : W \to V; y \mapsto \pi_F(Y_1)(y)$.
Then $\operatorname{Im}(\theta) = (0, 1] \not\supset V$.
Let $\theta^{\ast} : \overline{\boldsymbol{T}(X)} / E \to \overline{\boldsymbol{T}(Y_1, Y_2)} / F; f \mapsto f \circ \theta$ be the pull-back.
This $\theta^{\ast}$ is injective by the following argument.
Let $f, g \in \overline{\boldsymbol{T}(X)} / E$ such that $\theta^{\ast}(f) = \theta^{\ast}(g)$.
Since $f \circ \theta = \theta^{\ast}(f) = \theta^{\ast}(g) = g \circ \theta$ and $\operatorname{Im}(\theta) = (0, 1]$, the restrictions $f|_{(0, 1]}$ and $g|_{(0, 1]}$ coincide.
Thus, if $f = -\infty$, then so is $g$, and hence $f = g$.
If $f \not= -\infty$, then $g \not= -\infty$ and $f(0), g(0) \in \boldsymbol{R}$.
Since both $f$ and $g$ are continuous, $f(0) = g(0)$.
This means that $f|_V = g|_V$.
As $E = \boldsymbol{E}(V)$, we have $f = g$.
In conclusion, $\theta^{\ast}$ is injective.
Note that $\theta^{\ast}(\pi_E(X)) = \pi_F(Y_1)$.
Also $\pi_F(Y_2) \not\in \operatorname{Im}(\theta^{\ast})$.
In fact, if there exists $f \in \overline{\boldsymbol{T}(X)} / E$ such that $\theta^{\ast}(f) = \pi_F(Y_2)$, then $\theta^{\ast}(f)(t, 1/t)$ diverges to $\infty$ as $t \to 0$.
}
\end{ex}

By the following example, we know that we cannot drop the condition ``there exist $g_1, \ldots, g_m \in \overline{\boldsymbol{T}(X_1, \ldots, X_n)} / \boldsymbol{E}(V)$ such that for any $x \in V$, $\theta^{-1} (x) = (g_1(x), \ldots, g_m(x))$" in Proposition \ref{prop3}(5) and that the condition ``$V$ and $W$ are homeomorphic via $\theta$" is not a sufficient condition for ``$\overline{\boldsymbol{T}(X_1, \ldots, X_n)} / \boldsymbol{E}(V)$ and $\overline{\boldsymbol{T}(Y_1, \ldots, Y_m)} / \boldsymbol{E}(W)$ are isomorphic via $\theta^{\ast}$":

\begin{ex}
	\label{ex3}
\upshape{
Let $V := [0, 2] \subset \boldsymbol{R}$ and $W := [0, 1] \subset \boldsymbol{R}$.
Let $\pi_{\boldsymbol{E}(W)} : \overline{\boldsymbol{T}(Y)} \twoheadrightarrow \overline{\boldsymbol{T}(Y)} / \boldsymbol{E}(W)$ be the natural surjective $\boldsymbol{T}$-algebra homomorphism.
Let $\theta : W \to V; x \mapsto \pi_{\boldsymbol{E}(W)}(Y^{\odot 2})(x)$ and $\theta^{\ast} : \overline{\boldsymbol{T}(X)} / \boldsymbol{E}(V) \to \overline{\boldsymbol{T}(Y)} / \boldsymbol{E}(W)$ the pull-back.
This $\theta$ is a homeomorphism.
However there exists no element of $\overline{\boldsymbol{T}(X)} / \boldsymbol{E}(V)$ that induces $\theta^{-1}$.
Also $\pi_{\boldsymbol{E}(W)}(Y) \not\in \operatorname{Im}(\theta^{\ast})$ in this case.
Hence $\theta^{\ast}$ is not surjective.
}
\end{ex}

\end{document}